\documentclass[11pt]{amsart}

\usepackage{amssymb,amsfonts}
\usepackage{amsmath,amscd}
\usepackage[all,arc]{xy}
\usepackage{enumerate}
\usepackage{mathrsfs}
\usepackage[pdftex]{graphicx}
\usepackage[utf8]{inputenc}
\usepackage[T1]{fontenc}
\usepackage[usenames,dvipsnames]{color}
\usepackage[bookmarks=false]{hyperref}
\usepackage{dsfont}
\usepackage{tikz}
\usetikzlibrary{automata,positioning}
\usepackage{fullpage}

\theoremstyle{plain}
\newtheorem{thm}{Theorem}[section]

\newtheorem{prop}[thm]{Proposition}
\newtheorem{lem}[thm]{Lemma}

\theoremstyle{definition}

\newtheorem{example}[thm]{Example}

\theoremstyle{remark}

\newcommand{\bbA}{\mathbb{A}}

\newcommand{\bbC}{\mathbb{C}} 

\newcommand{\bbF}{\mathbb{F}}
\newcommand{\bbH}{\mathbb{H}} 

\newcommand{\bbP}{\mathbb{P}}
\newcommand{\bbQ}{\mathbb{Q}}
\newcommand{\bbR}{\mathbb{R}} 
\newcommand{\bbS}{\mathbb{S}}
\newcommand{\PSL}{\text{PSL}}
\newcommand{\PGL}{\text{PGL}}
\newcommand{\berA}{\bbA^{1,{\rm an}}_k}
\newcommand{\berP}{\bbP^{1,{\rm an}}_k}
\newcommand{\bbT}{\mathbb{T}}
\newcommand{\bbZ}{\mathbb{Z}} 

\newcommand*{\defeq}{\mathrel{\vcenter{\baselineskip0.5ex \lineskiplimit0pt
			\hbox{\scriptsize.}\hbox{\scriptsize.}}}%
	=}

\makeatletter
\let\c@equation\c@thm
\makeatother
\numberwithin{equation}{section}

\title{Basmajian's identity over non-Archimedean local fields}
\author{Yan Mary He}
\address{Department of Mathematics\\
	University of Oklahoma\\
	Norman, OK 73019}
\email{he@ou.edu}

\author{Chenxi Wu}
\address{Department of Mathematics, University of Wisconsin-Madison, Madison,
WI 53703}
\email{cwu367@wisc.edu}

\date{\today}

\begin{document}

\maketitle
\begin{abstract}
Let $\Sigma$ be a connected compact oriented surface with boundary and negative Euler characteristic. Let $k$ be a non-Archimedean local field. In this paper, we prove Basmajian's identity for projective Anosov representations $\rho \colon \pi_1\Sigma \to \PSL(d,k), d\ge 2$. Our series identity exhibits a drastic difference from all the Basmajian-type identities over the Archimedean fields $\mathbb{R}$ and $\mathbb{C}$. In particular, the series is a signed finite sum.
When $d=2$, we give a geometric proof of the identity using Berkovich hyperbolic geometry.
\end{abstract}

\section{Introduction}
If $\Sigma$ is a connected compact oriented hyperbolic surface with geodesic boundary $\partial\Sigma$, an orthogeodesic $\gamma$ on $\Sigma$ is a properly immersed geodesic arc that is perpendicular to $\partial\Sigma$ at both ends. In \cite{Bas}, Basmajian proved the following identity relating the sum of hyperbolic lengths of boundary geodesics to the lengths of orthogeodesics on $\Sigma$:
$$\ell(\partial\Sigma) = \sum_{\gamma} 2\log \coth^2 (\ell(\gamma)/2) $$
where the summation is taken over all orthogeodesics $\gamma$ on $\Sigma$. 

A hyperbolic structure on $\Sigma$ can be identified with a discrete faithful representation $\rho : \pi_1\Sigma \to \PSL(2,\bbR)$ acting on the hyperbolic plane as M\"obius transformations. For simplicity, we look for now surfaces with one boundary component. Up to
conjugation, we may assume that $\rho([\partial \Sigma])$ stabilizes the positive imaginary axis $L$. Orthogeodesics are in one-to-one correspondence with double cosets of $\pi_1(\partial \Sigma)$ in
$\pi_1\Sigma$ \cite{He, Vlamis}. For each nontrivial double coset $\pi_1(\partial \Sigma) w \pi_1(\partial \Sigma)$, the hyperbolic geodesic $w \cdot L$ corresponds to another boundary component of $\widetilde{\Sigma}$,
and the contribution to Basmajian's identity from this term is
$\log[\infty, 0, \rho(w)(\infty), \rho(w)(0)]$, where $[a,b,c,d]$ is the cross ratio on $\bbR\bbP^1$. Hence Basmajian's identity can be equivalently written as
$$\ell({\partial \Sigma}) = \sum_w \log[\infty, 0, \rho(w)(\infty), \rho(w)(0)]$$
where the summation is taken over all the double cosets $\pi_1(\partial \Sigma) w \pi_1(\partial \Sigma)$.

In this paper, instead of considering the Archimedean field $\bbR$, we consider projective Anosov representations $\rho : \pi_1\Sigma \to \PSL(d,k)$ where $d\ge2$ is an integer and $(k,|\cdot|)$ is a {\it non-Archimedean} local field with a non-trivial absolute value $|\cdot|$. Following Pozzetti-Sambarino-Wienhard \cite[Definition 3.1]{PSW}, we say that a homomorphism $\rho : \pi_1\Sigma \to \PSL(d,k)$ is {\it projective Anosov} if there exist positive constants $c, \mu$ such that for all $g \in \pi_1\Sigma$ we have
$$\frac{\sigma_{2}}{\sigma_1}(\rho(g)) \le ce^{-\mu|g|_w}$$
where $|g|_w$ is the word length of $g$ with respect to a fixed generating set. The numbers $\sigma_i$ are the {\it ratios of the semi-homothecy}; see Section \ref{sec_Anosov} for details. The main theorem of this paper is the following Basmajian-type identity for projective Anosov representations $\rho : \pi_1 \Sigma \to \PSL(d,k)$.

\begin{thm} \label{thm_realhigherBas}
Let $\Sigma$ be a connected compact oriented surface with boundary and negative Euler characteristic. Let $\alpha_1,\ldots, \alpha_m \in \pi_1\Sigma$ represent the free homotopy classes of boundaries $a_1,\ldots, a_m$. Let $d\ge 2$ be an integer and let $k$ be a non-Archimedean local field with a non-trivial absolute value $|\cdot|$. If $\rho: \pi_1\Sigma \to \PSL(d, k)$ is a projective Anosov representation, then we have
	\begin{equation} \label{eq_alg}
		\sum_{j=1}^m \ell_{\rho}(\alpha_j) = \sum_{j,q =1}^m \sum_{w \in \mathcal{L}_{j,q}} \log_v \left| C_{ \rho}(\alpha_j^+,\alpha_j^-;w\cdot \alpha_q^+,w \cdot \alpha_q^-) \right|
	\end{equation}
	where $\alpha_j^+,\alpha_j^- \in \partial \pi_1\Sigma$ are the attracting and repelling fixed points of $\alpha_j$, respectively. 
\end{thm} 
Here $\ell_{\rho}$ and $C_{\rho}$ denote certain length and cross ratio associated to $\rho$ respectively, and $\log_v$ is the logarithm function with base $p > 1$ being the cardinality of the residue field of $k$. The set $\mathcal{L}_{j,q}$ is a set of double coset representatives associated to boundary components $a_j$ and $a_q$. We note that if we put an ordering on a symmetric generating set of $\pi_1\Sigma$, each double coset $\pi_1(\partial \Sigma) w \pi_1(\partial \Sigma)$ has a {\it unique} representative $w$; namely, the shortest in word length and the largest with respect to the ordering.

It is worth pointing out that the series on the right hand side of \eqref{eq_alg} is in fact a {\it finite signed} sum. We will show in Section \ref{sec_examples} that there are projective Anosov represenations for which the right hand side series of Theorem \ref{thm_realhigherBas} consists of only positive terms, as well as representations for which the right hand side series consists of both positive and negative terms.
This is a major difference from the identities over Archimedean fields where the series involved is always an {\it infinite} sum of {\it positive} terms. 

Before we outline the strategy of the proof of the theorem, we mention related work on identities in hyperbolic geometry, higher Teichm\"uller theory and super Teichm\"uller theory. We note that all these identities are for representations over Archimedean fields. For identities on hyperbolic manifolds, we refer the reader to the survey paper \cite{BT} and the references therein. The first author extended Basmajian's identity to certain Schottky representations \cite{He}. Vlamis-Yarmola established Basmajian's identity for Hitchin representations \cite{Vlamis}. The first author established Basmajian's identity for hyperconvex Anosov representations \cite{He2}. Fanoni-Pozzetti obtained Basmajian-type inequalities for maximal representations \cite{PSW}. Moreover, there is recent progress in generalizing McShane's identity to higher Teichm\"uller theory \cite{LabMc, Zhe} and super Teichm\"uller theory \cite{HPZ}.

\subsection{Strategy of the proof}
We first prove Theorem \ref{thm_realhigherBas} for $d=2$ using Berkovich hyperbolic geometry and then give a proof for the general case. 

If $k$ is a non-Archimedian local field which is complete with respect to a non-trivial absolute value $|\cdot|$, then the topology on $k$ induced by the given absolute value is Hausdorff but it is totally disconnected. Berkovich's theory allows us to overcome this difficulty coming from the topology of $k$ and study analytic objects on $k$. The Berkovich affine line $\bbA^{1,an}_k$ is a locally compact, Hausdorff and path-connected topological space which contains $k$, with the topology induced by $|\cdot|$, as a dense subspace. The Berkovich projective line $\bbP^{1,an}_k$ is obtained by adding the point $\infty$ to compactify $\bbA^{1,an}_k$. Therefore, the Berkovich projective line $\bbP^{1,an}_k$ is a compact, Hausdorff and path-connected topological space which contains the projective line $\bbP^1(k)$ as a dense subspace. Moreover, $\bbP^{1,an}_k$ admits a tree structure and it is in fact an $\bbR$-tree. Based on the tree structure, the boundary of $\bbP_k^{1,an}$ is $\bbP^1(k)$ and the set $\bbH_k \defeq \bbP_k^{1,an} \setminus \bbP^1(k)$ is called the {\it Berkovic hyperbolic space}. 

The group $\PGL(2,k)$ of M\"obius transformations acts continuously on $\bbP^{1,an}_k$ in a natural way which is compatible with its action on $\bbP^1(k)$. This action preserves the Berkovic hyperbolic space $\bbH_k$, and moreover, there exists a metric $d_{\bbH}$ on $\bbH_k$ such that $\PGL(2,k)$ acts by isometries.

Let $\Sigma$ be a surface as in Theorem \ref{thm_realhigherBas}. A projective Anosov representation $\rho : \pi_1\Sigma \to \PSL(2,k)$ gives rise to an action of $\rho(\pi_1\Sigma)$ on $\bbH_k$. Since $\rho : \pi_1\Sigma \to \PSL(2,k)$ is projective Anosov, it is discrete and faithful, and for every $g \in \pi_1\Sigma$, the matrix $\rho(g)$ is hyperbolic as an element in $\PGL(2,k)$. The {\it limit set} of $\rho$ is the image of the Gromov boundary $\partial\pi_1\Sigma$ under the limit map $\zeta^1_{\rho} : \partial\pi_1\Sigma \to \bbP^1(k)$.

The general idea of proving Theorem \ref{thm_realhigherBas} when $d=2$ is analogous to Basmajian's original proof. Namely, for each boundary component $a_j$ of the surface, the hyperbolic length of a fundamental domain on the axis ${\rm Axis}(\alpha_j)$ of $\alpha_j$ is equal to the sum of the {\it projections} onto ${\rm Axis}(\alpha_j)$ of the translates $g\cdot {\rm Axis}(\alpha_j)$, where $g$ ranges over representatives of double cosets $\langle \alpha_j \rangle \backslash \pi_1\Sigma / \langle \alpha_q \rangle$. However, the key difference of the proof from that for Archimedean fields is that only finitely many projections are non-zero. This is due to the {\it non-Archimedean} feature of the absolute value of the field. Moreover, different from the case of identities over Archimedean fields, the length of the projections are {\it signed}, which may not always be positive as in the case of identities over Archimedean fields. 

The proof of Theorem \ref{thm_realhigherBas} in the general case does not use any geometric argument. Instead, for each boundary component $a_j$ of the surface, we interpret the length $\ell_\rho(\alpha_j)$ as the length of a fundamental domain on the real line $\mathbb{R}$ which plays the role of the axis ${\rm Axis}(\alpha_j)$ of $\alpha_j$. Then we show that the length of this fundamental domain is equal to the sum of lengths, one for each double coset in $\langle \alpha_j \rangle \backslash \pi_1\Sigma / \langle \alpha_q \rangle$. Similar to the case $d=2$, this sum is finite and some terms in the summation may be negative.
We will give examples of Anosov representations in Section \ref{sec_examples} to illustrate this phenomena.

\subsection{Organization of the paper} 
The paper is organized as follows. In Section 2, we review and prove preliminary results about Anosov representations over non-Archimedean local fields. In Section 3, we review and prove preliminary results about the Berkovich projective line and the Berkovich hyperbolic space. In Section 4, we prove Theorem \ref{thm_realhigherBas}. In Section 5, we give examples to demonstrate Theorem \ref{thm_realhigherBas}.

\subsection*{Acknowledgement} The first author would like to thank Charles Favre for organizing the arithmetic dynamics seminar at MSRI in Spring 2022 and for being inspiring.

\section{Anosov representations over non-Archimedean local fields}
In this section, we first give a quick introduction to non-Archimedean local fields in Section \ref{sec_field}. In Section \ref{sec_Anosov}, we give a brief overview of definitions and properties of Anosov representations $\rho : \Gamma \to \PGL(d,k)$ from a Gromov hyperbolic group $\Gamma$ into $\PGL(d,k)$ where $k$ is a non-Archimedean local field. We then discuss a {\it cross ratio} on the Gromov boundary of $\Gamma$ associated to a projective Anosov representation in Section \ref{subsec_cr}.

\subsection{Non-Archimedean local fields}\label{sec_field}
In this section, we collect the basics of non-Archimedean local fields. Standard references are \cite{Cas, Serre}.

Let $k$ be a field and let $k^{\times} = k \setminus \{0\}$. An {\it absolute value} on $k$ is a map $|\cdot| : k \to \bbR_{\ge 0}$ such that
$$|x|=0 \iff x=0; |xy|=|x|\cdot |y|, \forall x,y \in k; \text{ and } |x+y| \le |x|+|y|.$$
An absolute value on $k$ is called {\it non-Archimedean} if it also satisfies the ultrametric inequality $$|x+y| \le \max\{|x|,|y| \}.$$ Otherwise the absolute value is called {\it Archimedean}. 

The absolute value $|\cdot|$ on $k$ induces a metric on $k$; namely $d(x,y) \defeq |x-y|$ for any $x,y \in k$. A field is {\it complete} if every Cauchy sequence has a limit point in the field. It is a classical theorem that if $k$ is complete with respect to an Archimedean absolute value, then $k$ is isomorphic to $\bbR$ or $\bbC$ and the absolute value is equivalent to the usual absolute value. 

A {\it valuation} on a field $k$ is a group homomorphism $v : k^{\times} \to \mathbb R$ such that for all $x,y \in k$, we have $$v(x+y) \ge \min \{v(x),v(y)\}.$$ We may extend $v$ to a map $k \to \mathbb{R} \cup \{\infty\}$ by setting $v(0) = \infty$. For any $c \in (0,1)$, $|\cdot| \defeq c^{v(\cdot)}$ gives a non-Archimedean absolute value. The image of $v$ in $\mathbb R$ is called the {\it value group} and is denoted by $|k^{\times}|$. We say that $v$ is a {\it discrete valuation} if the value group is equal to $\mathbb Z$. The set $k^{\circ} \defeq \{ x \in k : v(x) \ge 0\}$ is the {\it valuation ring} of $k$ (with respect to $v$). It is a {\it discrete} valuation ring if $v$ is discrete.
If $k^{\circ}$ is a discrete valuation ring, any element $\pi \in k^{\circ}$ with $v(\pi) = 1$ is called a {\it uniformizer}. The ideal generated by $\pi$ is the unique maximal ideal of $k^{\circ}$. The {\it residue field} of $k$ is the field $k^{\circ}/\pi k^{\circ}$.

A {\it local field} is a field $k$ which is complete with respect to a nontrivial absolute value induced by a discrete valuation and whose residue field is finite.
An Archimedean local field is isomorphic to either $\mathbb R$ or $\mathbb{C}$.
A non-Archimedean local field is isomorphic to either a finite extension of the $p$-adic numbers $\bbQ_p$ if it has characteristic zero, or the field of formal Laurent series $\bbF_q((T))$ over a finite field $\bbF_q$ where $q$ is a power of a prime number $p$ if it has characteristic $p$.

\subsection{Projective Anosov representations} \label{sec_Anosov}
Let $\Gamma$ be a Gromov hyperbolic group. Let $(k,|\cdot|)$ be a non-Archimedean local field. Let $d \ge 2$ be an integer and let $p$ be an integer with $1 \le p \le d-1$. Following Pozzetti-Sambarino-Wienhard \cite[Definition 3.1]{PSW} which generalizes the work of Bochi-Potrie-Sambarino \cite{BPS}, we say that a homomorphism $\rho : \Gamma \to \PGL(d,k)$ is {\it $p$-Anosov} if there exist positive constants $c, \mu$ such that for all $\gamma \in \Gamma$ we have
$$\frac{\sigma_{p+1}}{\sigma_p}(\rho(\gamma)) \le ce^{-\mu|\gamma|_w}$$
where $|\gamma|_w$ is the word length of $\gamma$ with respect to a fixed generating set. The numbers $\sigma_i$ are defined as follows. Let $V$ be a $k$-vector space (e.g. $V = k^d$). Given a $k$-norm $||\cdot||_{V, k}$ on $V$, we say that $g \in {\rm GL}(V,k)$ is a {\it semi-homothecy} if there exists a $g$-invariant $k$-orthogonal decomposition $V = V_1 \oplus...\oplus V_m$ and $\sigma_1,...,\sigma_m \in \bbR_{>0}$ such that for every $i=1,...,m$ and every $v_i \in V_i$ one has $$||gv_i||_{V, k} = \sigma_i ||v_i||_{V, k}.$$ The numbers $\sigma_i$ are called the {\it ratios of the semi-homothecy}. A $1$-Anosov representation is called a {\it projective Anosov} representation.

An important property of Anosov representations is that they admit equivariant transverse limit maps given by the following proposition. 

\begin{prop}[{\cite[Proposition 3.5]{PSW}}]
If $\rho : \Gamma \to \PGL(d,k)$ is $p$-Anosov, then for any geodesic ray $\{\gamma_n\}_{n \ge 1}$ in (the Cayley graph of) $\Gamma$ with endpoint $x$ in the Gromov boundary $\partial\Gamma$, the limits
	$$\zeta^p_{\rho}(x) \defeq \lim_{n \to \infty} U_p(\rho(\gamma_n))$$
	$$\zeta^{k-p}_{\rho}(x) \defeq \lim_{n \to \infty} U_{n-p}(\rho(\gamma_n))$$
	exist and do not depend on the ray. Moreover, they define continuous $\rho$-equivariant transverse maps $\zeta^p : \partial\Gamma \to \mathcal{G}_p(k^d)$ and $\zeta^{d-p} : \partial\Gamma \to \mathcal{G}_{d-p}(k^d)$, where  $\mathcal{G}_j(k^d)$ denotes the Grassmannian $j$-planes in $k^d$ for $j=p, n-p$. Here $U_p(\rho(\gamma))$ denotes the Cartan attractor of $\rho(\gamma)$ (see \cite[ p. 9]{PSW}), i.e.
	$$U_{p}(\rho(\gamma)) = \beta \cdot (e_1 \oplus ... \oplus e_{p})$$ with $\rho(\gamma) = \beta a\beta'$ coming from the Cartan $KAK$ decomposition of $GL(d, k)$ and $\{e_j\}_{j=1}^d$ being the standard basis for the vector space $k^d$.
\end{prop}

If $\rho : \Gamma \to \PGL(d,k)$ is projective Anosov, then the set $\zeta_{\rho}^1(\partial\Gamma) \subset \bbP(k^d)$ is called the {\it limit set} of $\rho$. 

\subsection{Cross ratios} \label{subsec_cr}
Let $\Sigma$ be a connected compact oriented surface with boundary and negative Euler characteristic. In this subsection, we define a cross ratio for four points on the Gromov boundary $\partial\pi_1\Sigma$ of the hyperbolic group $\pi_1\Sigma$ which is a natural generalization of Labourie's cross ratio; see \cite{Lab07, LabMc}. This is achieved by using the transverse limit maps of a projective Anosov representation which take points of $\partial\pi_1\Sigma$ into the projective spaces, and then taking the cross ratio considered by Falbel-Guilloux-Will \cite{FGW} that we introduce now.

Let $k$ be any field. Let $d \ge 2$ be an integer and denote by $k^{d*}$ the dual of the vector space $k^{d}$. Consider two non-empty sets $\Omega \subset \bbP (k^{d})$ and $\Lambda \subset \bbP(k^{d*})$ such that $$\varphi(\omega) \neq 0$$ for any $\omega \in \Omega$ and $\varphi \in \Lambda$. In other words, $\Omega$ is disjoint from all hyperplanes defined by points in $\Lambda$.

Let $(\varphi, \varphi', \omega, \omega')\in \Lambda^2 \times \Omega^2$ with $\omega \neq \omega'$ so that $\omega$ and $\omega'$ span a projective line $(\omega\omega')$. In \cite{FGW}, Falbel-Guilloux-Will considered the following cross ratio
$$[[\varphi, \varphi', \omega, \omega']] \defeq [\varphi_{\omega,\omega'}, \varphi'_{\omega,\omega'}, \omega, \omega']$$ where $\varphi_{\omega,\omega'}$ and $\varphi'_{\omega,\omega'}$ are the points $\text{ker}(\varphi) \cap (\omega\omega')$ and $\text{ker}(\varphi') \cap (\omega\omega')$ in $(\omega\omega')$ respectively, and $[w_1,w_2,w_3,w_4]$ is the usual cross ratio for four distinct points in $\bbP^1(k) = (\omega\omega')$, namely $[w_1,w_2,w_3,w_4] = (w_1-w_3)(w_2-w_4)/(w_1-w_4)(w_2-w_3)$.

The cross ratio satisfies the following properties.
\begin{lem} \label{lem_cr_FGW}
	Given $(\varphi, \varphi', \omega, \omega')\in \Lambda^2 \times \Omega^2$, the cross ratio $[[\varphi, \varphi', \omega, \omega']]$ satisfies 
	\begin{equation} \label{eq_crossratio}
		[[\varphi, \varphi', \omega, \omega']] = \frac{\bar{\varphi}(\bar{\omega})\bar{\varphi}'(\bar{\omega}')}{\bar{\varphi}(\bar{\omega}')\bar{\varphi}'(\bar{\omega})}
	\end{equation}
	where $\bar{\varphi}, \bar{\varphi}', \bar{\omega}$ and $\bar{\omega}'$ are any lifts of $\varphi, \varphi', \omega$ and $\omega'$. Moreover, if $\omega, \omega', \omega''$ are three points in $\Omega$ and $\varphi, \varphi', \varphi''$ are three points in $\Lambda$, we have
	\begin{enumerate}
		\item $[[\varphi, \varphi', \omega, \omega']] = 1$ if and only if $\varphi = \varphi'$ or $\omega = \omega'$,
		\item $[[\varphi, \varphi', \omega, \omega']] = [[\varphi, \varphi', \omega, \omega'']]\cdot [[\varphi, \varphi', \omega'', \omega']]$, and
		\item $[[\varphi, \varphi', \omega, \omega']] = [[\varphi, \varphi', \omega', \omega]]^{-1}$.
	\end{enumerate}
\end{lem}
\begin{proof}
	We give a sketch of the proof here and refer the interested reader to {\cite[Lemma 2.3, Proposition 2.4]{FGW}}. Equation (\ref{eq_crossratio}) is obvious if $\varphi = \varphi'$.
	Suppose $\varphi \neq \varphi'$, we choose a basis $\{e_m\}_{m=1}^{n}$ of the vector space $k^n$ such that the lifts $\bar{\varphi}, \bar{\varphi}', \bar{\varphi}_{\omega,\omega'}$ and $\bar{\varphi}'_{\omega,\omega'}$ (in $k^{n*}$ and $k^{n}$) are as follows:
	$$\bar{\varphi} = e_1^*, \bar{\varphi}' = e_2^*, \bar{\varphi}_{\omega,\omega'} = e_2 + \bar{u}, \bar{\varphi}'_{\omega,\omega'} = e_1 + \bar{u}'$$
	where $\bar{u}$ and $\bar{u}'$ are vectors in Span$(e_3, \cdots, e_n)$. Note that $\bar{\varphi}_{\omega,\omega'}$ and $\bar{\varphi}'_{\omega,\omega'}$ form a basis for the plane spanned by $\bar{\omega}$ and $\bar{\omega}'$. Therefore $\bar{\omega}$ and $\bar{\omega}'$ can be written as
	$$\bar{\omega} =  \lambda'\bar{\varphi}'_{\omega,\omega'} + \lambda\bar{\varphi}_{\omega,\omega'} \text{ and } \bar{\omega}' = \mu'\bar{\varphi}'_{\omega,\omega'} + \mu\bar{\varphi}_{\omega,\omega'}.$$
	
	With respect to the coordinate basis $\{\bar{\varphi}'_{\omega,\omega'}, \bar{\varphi}_{\omega,\omega'}\}$, in the projective line $(\omega\omega')$, $\varphi_{\omega,\omega'}$ has projective coordinate $[1, \infty = 1/0]$, $\varphi'_{\omega,\omega'} = [1, 0]$, $ \omega = [\lambda',\lambda] = [1; \lambda/\lambda']$ and $ \omega' = [\mu',\mu] = [1; \mu/\mu']$.
	
	Therefore, we have
	$$\frac{\bar{\varphi}(\bar{\omega})\bar{\varphi}'(\bar{\omega}')}{\bar{\varphi}(\bar{\omega}')\bar{\varphi}'(\bar{\omega})} = \frac{\mu\lambda'}{\mu'\lambda} =[ \varphi_{\omega,\omega'},\varphi'_{\omega,\omega'}, \omega, \omega']$$
	proving equation (\ref{eq_crossratio}).
	
	Identities (1)-(3) can be verified by straightforward computations.
\end{proof}

Now we define a cross ratio for four points on the boundary $\partial\pi_1\Sigma$. Let $\rho : \pi_1\Sigma \to \PGL(d, k)$ be a projective Anosov representation with limit maps $\zeta^1$ and $\zeta^{d-1}$. We denote by $(\partial\pi_1\Sigma)^{4*}$ the set
$$(\partial\pi_1\Sigma)^{4*} \defeq \{(x,y,u,v) \in (\partial\pi_1\Sigma)^{4} : x \neq v, y \neq u\}.$$
The {\it cross ratio} is a continuous function $C_{\rho} : (\partial\pi_1\Sigma)^{4*} \to \bbR$ given by
$$C_{\rho}(x,y,u,v) = [[\zeta^{d-1}(x), \zeta^{d-1}(y), \zeta^1(u), \zeta^1(v)]].$$
We note that if $\rho : \pi_1\Sigma \to \PSL(d,\bbR)$ is Hitchin, then $C_{\rho}$ is Labourie's cross ratio.


\begin{lem} \label{lem_pos}
	Let $\rho : \pi_1\Sigma \to \PGL(d,k)$ be a projective Anosov representation. The cross ratio $C_{\rho}$ satisfies the following properties
	\begin{enumerate}
		\item $C_{\rho}(x,y,u,v) = 0 \iff x = u$ or $y=v$.
		\item $C_{\rho}(x,y,u,v) = 1 \iff x = y$ or $u=v$.
		\item $C_{\rho}(x,y,u,v) = C_{\rho}(x,y,u,v')C_{\rho}(x,y,v',v)$
		\item $C_{\rho}(x,y,u,v) = C_{\rho}(x,y,v,u)^{-1}$
	\end{enumerate}
\end{lem}
\begin{proof}
	(1)-(4) can be checked directly by definition and Lemma \ref{lem_cr_FGW}.
\end{proof}

Given a projective Anosov representation $\rho: \pi_1\Sigma \to \PGL(d,k)$, the {\it translation length} of any nontrivial element $\gamma \in \pi_1\Sigma$ is defined as
$$\ell_{\rho}(\gamma) \defeq \log_v \left|\frac{\lambda_1(\rho(\gamma))}{\lambda_{d}(\rho(\gamma))}\right|$$
where $\lambda_1(\rho(\gamma))$ and $\lambda_d(\rho(\gamma))$ are the eigenvalues of maximum and minimum absolute value of $\rho(\gamma)$, respectively, and $\log_v$ is the logarithm function with base $q > 1$ being the cardinality of the residue field of $k$. The {\it period} of a nontrivial element $\gamma \in \pi_1\Sigma$ is defined as $$P_{\rho}(\gamma) \defeq \log_v \left| C_{\rho} (\gamma^+,\gamma^-,x,\gamma x )\right|$$ where $x$ is any point in $\partial\pi_1\hat{\Sigma} \setminus \{\gamma^+,\gamma^-\}$.

If $k = \bbR$ or $\bbC$, then translation lengths and periods become the usual notion of (complex) lengths and (complex) periods. Moreover, similar to the case $k = \bbR$ proved by Labourie \cite[Proposition 5.8]{Lab07}, the period of a nontrivial element equals its translation length.
\begin{lem}\label{lem_2.8}
	For any nontrivial element $\gamma \in \pi_1\Sigma$, $P_{\rho}(\gamma) = \ell_{\rho}(\gamma)$.
\end{lem}
\begin{proof}
	The proof follows from the same calculation which is carried in details in \cite[ Proposition 5.8]{Lab07}.
\end{proof}

%
%
%

\section{The Berkovich projective line $\berP$} \label{sec_pre}
Let $(k,|\cdot|)$ be a complete metrized non-Archimedean field. In this section, we give a brief overview of the Berkovic projective line $\berP$, the Berkovic hyperbolic space $\bbH_k$ and the group $\PGL(2,k)$ that we will use in the next section. Our exposition follows closely from \cite{DF} and \cite{BR}.

\subsection{The Berkovich projective line and the Berkovich hyperbolic space}
Let $(k, |\cdot|)$ be a complete metrized non-Archimedean field. Recall that $k^{\times} = k \setminus \{0\}$ and $|k^{\times}|$ is the value group of $k$. Let $\bbP^1_k$ be the projective line over $k$, viewed as an algebraic variety with Zariski toplogy, and by $\bbP^1(k)$ the set of $k$-points of $\bbP^1_k$ which as a set is $k \cup \{ \infty \}$. The Berkovich analytification $\berP$ of $\bbP^1_k$ is defined as follows. Recall that a {\it multiplicative semi-norm} on a ring $A$ is a function $[~]_x : A \to \bbR_{\ge 0}$ such that 
$$[0]_x = 0, [1]_x = 1, [fg]_x = [f]_x\cdot [g]_x \text{ and } [f+g]_x \le [f]_x + [g]_x $$ for all $f,g \in A$.
We first define the Berkovich analytification $\bbA^{1,{\rm an}}_k$ of the affine line $\bbA^1_k$ as the set of multiplicative semi-norms on $k[Z]$ whose restriction to  $k$ coincides with $|\cdot|$, endowed with the topology of pointwise convergence. More specifically, given a point $x \in \bbA^{1,{\rm an}}_k$ and a polynomial $P \in k[Z]$, the value of the semi-norm defined by $x$ on $P$ is denoted by $|P(x)| \in \bbR_{\ge 0}.$ The {\it Gauss norm} $\sum_ia_iZ^i \mapsto {\rm max} |a_i|$ defines a point $x_g$ in $\bbA^{1, {\rm an}}_k$ which is called the {\it Gauss point}. The Berkovich projective line $\berP$ can be defined topologically as the one-point compactification of $\bbA^{1,{\rm an}}_k$; namely, we have $\berP \defeq \bbA^{1,{\rm an}}_k \cup \{ \infty \}$.

Points in $\berP$ correspond to balls in $\bbP^1(k)$. Recall that a closed (resp. open) ball in $k$ of radius $R \ge 0$ is defined as $$\overline{B}(z_0,R) = \{ z \in k : |z-z_0| \le R \}$$ (resp. $B(z_0,R) = \{ z \in k : |z-z_0| < R \}$). A ball in $\bbP^1(k)$ is either a ball in $k$ or the complement of a ball in $k$. Any closed (or open) ball $B \subset \bbP^1(k)$ determines a point $x_B \in \berP$. Indeed, if $B$ or its complement is a singleton $\{z\}$, then the point $x_B$ is the rigid point associated to $z$. If $B$ is a (closed or open) ball of finite radius in $k$, then $x_B$ is the point in $\berA$ corresponding to the semi-norm $$|P(x_B)| \defeq \sup_B|P|.$$ If the complement of $B$ is a ball of finite radius, then we set $$|P(x_B)| \defeq \sup_{k \setminus B}|P|.$$ The Gauss point $x_g$ corresponds to $x_{\bar{B}(0,1)}$. 
Moreover, the Berkovich projective line $\berP$ is a (profinite) $\bbR$-tree in the sense that it is uniquely path connected; namely, for any two points $x,y\in \berP$, there is a unique segment $[x,y] \subset \berP$.

A point in $\berP$ is said to be {\it rigid} if the multiplicative semi-norm has a non-trivial kernel. There is a canonical bijection between the set of closed points of the $k$-scheme $\bbP^1_k$ and the set of rigid points in $\berP$. To see this, let $z$ be a point in a finite extension of $k$. Then the semi-norm $|\cdot|_z$ defined by $P \mapsto |P(z)|$ is a rigid point in $\bbA^{1,{\rm an}}_k$. Therefore, $\bbP^1(k)$ naturally embeds in $\berP$ as a set of rigid points.

Now suppose $(k, |\cdot|)$ is a complete metrized non-Archimedean field which is algebraically closed. 
Points in $\berP$ can be classified by their {\it types}. Points in $\berP$ corresponding to balls of diameter zero in $k$ are called {\it type I} points of $\berP$. These points are in bijection with rigid points in the Berkovich projective space $\berP$. Points in $\berP$ corresponding to balls of diameter ${\rm diam}(B) \in |k^{\times}|$ are said to be of {\it type II}. Points in $\berP$ corresponding to balls of diameter ${\rm diam}(B) \notin |k^{\times}|$ are said to be of {\it type III}.

More generally, points in $\berP$ are in bijection with (equivalence classes of) decreasing sequences of balls. If $k$ is spherically complete (i.e. every decreasing interseciton of balls is non-empty), then $\berP$ consists of type I, II or III points. A local field is spherically complete. In general, $\berP$ could contain {\it type IV} points corresponding to decreasing sequences of balls whose intersection is empty. For each $i=1,...,4$, the set of type $i$ points, if it is non-empty, is dense in $\berP$.

Types of points in $\berP$ are related to the $\bbR$-tree structure of $\berP$ as follows. Type I and IV points are vertices of the $\bbR$-tree $\berP$ which have only one branch; namely, they are the end points of the tree. Type II points are vertices of $\berP$ which have at least three branches. Finally type III points are the vertices of $\berP$ which have exactly two branches; that is, the regular points.

As a set, the Berkovich hyperbolic space $\bbH_k$ is defined as $$\bbH_k \defeq \berP \setminus \bbP^1(k).$$ It is a proper subtree of $\berP \setminus \bbP^1(k)$ which is neither open nor closed. Moreover, it contains no rigid points.

There is a partial order on $\berA$. If $x,y \in \berA$, we say that $x \le y$ if $$[f]_x \le [f]_y \text{ for all } f \in k[Z].$$ In terms of balls, if $x,y$ are of type I, II or III, we say that $x \le y$ if the ball corresponding to $x$ is contained in the ball corresponding to $y$. For any pair of points $x, y \in \berA$, there exists a unique least upper bound $x \cup y \in \berA$ with respect to this partial order.

We now define a metric $d_{\bbH}$ on $\bbH_k$. If $x,y \in \bbH_k$ with $x \le y$, then we define $$d_{\bbH}(x,y) \defeq \log_v\frac{{\rm diam}(y)}{{\rm diam}(x)}.$$
More generally, if $x,y \in \bbH_k$, we define
$$d_{\bbH}(x,y) = d_{\bbH}(x,x\cup y) + d_{\bbH}(y, x\cup y).$$ It is straightforward to check that $d_{\bbH}$ defines a metric on $\bbH_k$. We can extend $d_{\bbH}$ on $\berP$ by setting $d_{\bbH}(x,y) = \infty$ if $x \in \bbP^1(k)$ and $y \in \berP$ with $x \neq y$, and $d_{\bbH}(x,y) = 0$ if $x = y$.

\begin{thm} [{\cite[Proposition 2.29]{BR}}]
The space $(\bbH_k, d_{\bbH})$ is an $\bbR$-tree. The group $\PGL(2,k)$ acts as isometries on $(\bbH_k, d_{\bbH})$.
\end{thm}

Now suppose $k$ is not algebraically closed. Let $\bar{k}^a$ be the completion of an algebraic closure of $k$. Then $\berP$ is homeomorphic to the quotient of $\bbP^{1,{\rm an}}_{\bar{k}^a}$ by the Galois group ${\rm Gal}(\bar{k}^a/k)$. Let $\pi_{\bar{k}^a/k} : \bbP^{1,{\rm an}}_{\bar{k}^a} \to \berP$ be the projection map. The action of the Galois group ${\rm Gal}(\bar{k}^a/k)$ on $\bbP^{1,{\rm an}}_{\bar{k}^a}$ preserves types of points. Therefore, we can define the type of a point $x \in \berP$ by the type of the point $\pi_{\bar{k}^a/k}^{-1}(x) \in \bbP^{1,{\rm an}}_{\bar{k}^a}$. We note that in general some type I point in $\berP$ may {\it not} be rigid.

The action of the Galois group ${\rm Gal}(\bar{k}^a/k)$ preserves the diameter of balls in $\bar{k}^a$. Therefore ${\rm Gal}(\bar{k}^a/k)$  acts by isometries on $(\bbH_{\bar{k}^a}, d_{\bbH})$. The set $\widetilde{\bbH}_k \subset \bbH_{\bar{k}^a}$ of fixed points of this action turns out to be the convex hull $\overline{{\rm Conv}(\bbP^1(k))}$ of $\bbP^1(k)$ in $\bbH_{\bar{k}^a}$; see \cite[Lemma 1.6]{DF}. Hence, in $\bbP_{\bar{k}^a}^{1,{\rm an}}$, we define $$\widetilde{\bbH}_k \defeq \overline{{\rm Conv}(\bbP^1(k))} \setminus \bbP^1(k).$$ Then $(\widetilde{\bbH}_k, d_{\bbH})$ is a complete metric $\bbR$-tree as shown in the previous subsection. Define
$$\bbH_k \defeq \pi_{\bar{k}^a/k}(\widetilde{\bbH}_k)$$ and endow $\bbH_k$ with the metric $d_{\bbH}$ so that the projection map $\pi_{\bar{k}^a/k} : (\widetilde{\bbH}_k, d_{\bbH}) \to (\bbH_k, d_{\bbH})$ is an isometry. We call $(\bbH_k, d_{\bbH})$ the Berkovic hyperbolic space of $k$.
	
The action of $\PGL(2,k)$ on $\bbP^{1,{\rm an}}_{\bar{k}^a}$ preserves types of points; see \cite[Proposition 2.15]{BR}. Therefore, it also preserves types of points in $\berP$. Morevoer, $\PGL(2,k)$ acts on $(\bbH_k, d_{\bbH})$ by isometries as it does on $(\widetilde{\bbH}_k, d_{\bbH})$.


\subsection{Elements of $\PGL(2,k)$} \label{subsec_PGL}
The first proposition of this subsection gives a classification of elements in $\PGL(2,k)$.
\begin{prop}[{\cite[Proposition 2.2]{DF}}]
Let $A \in \PGL(2,k)$ and $A \neq {\rm id}$. Then exactly one of the following holds.
\begin{enumerate}
\item $|{\rm tr}(\gamma)| > 1$: then $A$ is diagonalizable over $k$ and has one attracting (resp. repelling) fixed point $x_{att} \in \bbP^1(k)$ (resp. $x_{rep} \in \bbP^1(k)$). Furthermore, for any $x \neq x_{rep}$ in $\berP$, the sequence $A^n \cdot x$ converges to $x_{att}$ as $n \to \infty$. In this case, $A$ is said to be {\it hyperbolic}.
\item $|{\rm tr}(\gamma)| \le 1$: then $A$ admits a fixed point in $\bbH_k$. In this case, $A$ is said to be {\it elliptic}. More precisely if
\begin{itemize}
\item ${\rm tr}^2(\gamma) = 4$: then $A$ is not diagonalizable, and is conjugate in $\PGL(2,k)$ to $z \mapsto z+1$; thus it fixes a segment $[x,y] \subset \berP$ where $x$ (resp. $y$) is a type II (resp. type I) point belonging to the $\PGL(2,k)$ orbit of the Gauss point. In this case, $A$ is said to be {\it parabolic}.

\item ${\rm tr}^2(\gamma) \neq 4$: then in a field extension $K/k$ of degree at most 2, the matrix $A$ is diagonalizable. In addition, $A$ is conjugate to an element in $\PGL(2, k^{\circ})$ and fixes a type II point in $\bbH_k$. In this case, $A$ is said to be {\it strictly elliptic}.
\end{itemize}
\end{enumerate}
\end{prop}

If $\gamma \in {\rm SL}(2,k)$, we define the norm $||\gamma||$ of $\gamma$ to be the largest absolute value of entries in $\gamma$. The next proposition gives the Cartan KAK decomposition of elements in $\PGL(2,k)$.
\begin{prop}[Cartan decomposition, {\cite[Proposition 2.4]{DF}}]
Let $\gamma \in {\rm SL}(2,k)$. Then there exist $m,n \in {\rm SL}(2, k^{\circ})$ and $a = {\rm diag}(\lambda, \lambda^{-1})$ where $\lambda \in k$ and $|\lambda| \ge 1$ such that $\gamma = man$. Furthermore $||\gamma|| = ||a|| = |\lambda|$.
\end{prop}

\subsection{Projections and cross ratios}
In this subsection, $(k,|\cdot|)$ is a non-Archimedean local field with a non-trivial absolute value $|\cdot|$.
Let $\alpha \in \bbP^1(k)$ and consider the geodesic $[0,\infty] \subset \berP$ from $0 \in \bbP^1(k)$ and $\infty \in \bbP^1(k)$. We say that $\beta \in [0,\infty]$ is the {\it projection} of $\alpha$ onto $[0,\infty]$ if $\beta$ is the unique point in $[0,\infty]$ such that there exists a geodesic segment $[\beta, \alpha] \subset \berP$. The following lemma gives a formula for $\beta$ in terms of balls in $\bbP^1(k)$.
\begin{lem}\label{lem_proj_ball}
	Let $\alpha \in \bbP^1(k)$. Then the projection of $\alpha$ onto the geodesic $[0,\infty] \subset \berP$ is the ball $B(0, |\alpha|)$. 
\end{lem}
\begin{proof}
	The point $\alpha \in \bbP^1(k)$ corresponds to the ball $B(\alpha, 0)$ in $k$, which is disjoint from the ball $B(0,0)$. Therefore the point of projection of $\alpha$ onto $[0,\infty]$ is $B(0,|\alpha-0|)$.
\end{proof}

The next lemma states that for $a,b \in \bbP^1(k)$ satisfying $0< |a| < |b|$, the quantity $\log_v \left|[\infty, 0; |a|,|b|]\right|$ is equal to the hyperbolic distance between the projection of $a$ and $b$ onto the geodesic $[0,\infty]$.
\begin{lem} \label{lem_projcr}
If $a,b \in \bbP^1(k)$ satisfy $0< |a| < |b|$, then we have $$d_{\bbH}(B(0,|a|), B(0,|b|)) = \log_v \left|[\infty, 0; |a|,|b|]\right|$$ where $[u,v;x,y]$ is the cross ratio on $\bbP^1(k)$.
\end{lem}
\begin{proof}
	If $0<|a|<|b|$, the geodesic segment $[B(0,|a|), B(0,|b|)] \subset [0, \infty] \subset \berP$ is isometric to the interval $[\log_v |a|, \log_v |b|] \subset \bbR$ equipped with the Lebesgue measure. Therefore, we have
	$$d_{\bbH}(B(0,|a|), B(0,|b|)) =\log_v |b| - \log_v |a| =  \log_v \left|[\infty, 0; |a|,|b|]\right|.$$
\end{proof}


\section{Non-Archimedean Basmajian's identity}
In this section, we first give a geometric proof of Theorem \ref{thm_realhigherBas} when $d=2$ using Berkovic hyperbolic geometry in Section \ref{sec_pf_main_thm_d=2}.
We then prove Theorem \ref{thm_realhigherBas} in Section \ref{sec_pf_main_thm}.

\subsection{Proof of Theorem \ref{thm_realhigherBas} for $d=2$ using Berkovic hyperbolic geometry} \label{sec_pf_main_thm_d=2}
Let $\Sigma$ be a connected compact oriented surface with boundary $\partial\Sigma$ and negative Euler characteristic. Let $\alpha_1,\ldots, \alpha_m \in \pi_1\Sigma$ represent the free homotopy classes of boundaries $a_1,\ldots, a_m$. Each boundary $a_j$ is oriented so that the surface is on the right. Fix a finite area hyperbolic structure on $\Sigma$ so that $\partial\Sigma$ is totally geodesic. Let $V \subset \bbH^2$ be the universal cover of $\Sigma$ so that $\partial\pi_1\Sigma = \partial_{\infty}V = \overline{V} \cap \bbS^1$. Fix a boundary component $\alpha_j$ of $\Sigma$ and let $\alpha_j^+,\alpha_j^- \in \partial \pi_1\Sigma$ be the attracting and repelling fixed points of $\alpha_j$, respectively. We consider the interval $(\alpha_j^+, \alpha_j^-)$ as a subset of $\bbS^1$ with the given orientation.

We first observe that $\bbS^1 \setminus \partial_{\infty}V$ is a disjoint union of intervals of the form 
\begin{equation} \label{eq_I_beta}
\tilde{I}_{\beta} = (\beta^+, \beta^-)
\end{equation}
where $\beta  = g\alpha_kg^{-1}$ for some boundary element $\alpha_k$ and $g \in \pi_1\Sigma$, since the endpoints $\beta^-$ and $\beta^+$ are the endpoints of lifts of boundary elements.
\begin{lem}\label{lem_coset_beta}
Set $H_j \defeq \langle \alpha_j \rangle, j=1,\ldots,m$. Each interval $\tilde{I}_{\beta}$ correspond to a coset $gH_q$ in $\bigsqcup_{1 \le q \le m} \pi_1\Sigma / H_q.$
\end{lem}
\begin{proof}
If $\tilde{I}_{g_1\alpha_{l'}g_1^{-1}} = \tilde{I}_{g_2\alpha_lg_2^{-1}}$, then we know that
$$g_1\alpha_{l'}^{\pm} = g_2\alpha_l^{\pm}$$ and therefore $(g_2^{-1}g_1)\alpha_{l'}^{\pm} = \alpha_l^{\pm}$. Since $\alpha_{l'}$ and $\alpha_l$ are boundary elements, it follows that ${l'}=l$ and therefore $g_2^{-1} g_1 \in H_l$.
\end{proof}

Let $\rho : \pi_1\Sigma \to \PSL(2,k)$ be projective Anosov. Consider the action of $\rho(\pi_1\Sigma)$ on the Berkovic space $\bbP^{1,an}_k$.
We conjugate $\rho$ so that $\zeta^1(\alpha_j^-) = 0$ and $\zeta^1(\alpha_j^+) = \infty$. Therefore $[0,\infty] \subset \bbH_k$ is the axis of $\alpha_j$. Let $P \colon \bbP^1(k) \to [0,\infty]$ be the projection onto the axis $[0,\infty]$ given by $P(x) = B(0,|x|)$ by Lemma \ref{lem_proj_ball}. Define the map $\Psi_\rho \colon (\alpha_j^+, \alpha_j^-) \to (0,\infty)$ by
$$\Psi_\rho(x) \defeq P \circ \zeta^1(x).$$
By definition, for any $x \in (\alpha_j^+,\alpha_j^-) \subset \mathbb{S}^1$, $\Psi_\rho(x)$ is the projection of the point $\zeta^1(x) \in \mathbb{P}^1(k)$ in the limit set onto the axis $(0,\infty) \subset \mathbb{H}_k$ of $\alpha_j$.

\begin{lem} \label{lem_d=2_1}
Let $x \in (\alpha_j^+, \alpha_j^-) \subset \bbS^1$. Then we have $\ell_{\rho}(\alpha_j) = d_{\bbH}(\Psi_\rho(\alpha_j \cdot x), \Psi_\rho(x)).$
\end{lem}
\begin{proof}
By definition of $\Psi_\rho$ and Lemma \ref{lem_projcr}, we have
\begin{align*}
	d_{\bbH}(P\circ \zeta^1(\alpha_j \cdot x), P\circ \zeta^1(x))& = \log_v|\zeta^1(\alpha_j \cdot x)| - \log_v|\zeta^1(x)|\\
	& =  \log_v \left| C_{\rho}(\alpha_j^+, \alpha_j^-; x, \alpha_j \cdot x) \right|\\
	&=\ell_{\rho}(\alpha_j).
\end{align*}
The proof is completed.
\end{proof}

For each interval $\tilde{I}_\beta$, define $\hat{I}_{\beta} \defeq (\Psi_\rho(\beta^+), \Psi_\rho(\beta^-)) \subset [0,\infty]$. 

\begin{lem}\label{lem_proj_0} 
All but finitely many intervals $\tilde{I}_{\beta}$ satisfy $\Psi_\rho(\beta^-) = \Psi_\rho(\beta^+)$. 
\end{lem}
\begin{proof}
By definition of $\Psi_\rho$ and Lemma \ref{lem_projcr}, we have
\begin{align*}
d_{\bbH}(P\circ \zeta^1(\beta^+), P\circ \zeta^1(\beta^-))& = \log_v|\zeta^1(\beta^+)| - \log_v|\zeta^1(\beta^-)|\\
& =  \log_v \left| C_{\rho}(\alpha_j^+, \alpha_j^-; \beta^+, \beta^-) \right|.
\end{align*}
Since $\zeta^1 \colon \partial \pi_1\Sigma \to \bbP^1(k)$ is continuous, if $\beta^+$ and $\beta^-$ are close, the cross ratio $C_{\rho}(\alpha_j^+, \alpha_j^-; \beta^+, \beta^-)$ is close to 1. Then the non-Archimedean absolute value gives $|C_{\rho}(\alpha_j^+, \alpha_j^-; \beta^+, \beta^-)| = 1$. Therefore there are only finitely many $\widetilde{I}_\beta$ with $\Psi_\rho(\beta^+) \neq \Psi_\rho(\beta^-)$.
The proof is completed.
\end{proof}

By Lemma \ref{lem_proj_0}, there are only finitely many $\hat{I}_\beta$ which are not a singleton. Moreover, we note that $\hat{I}_\beta$ may or may not have the same orientation as $(0,\infty)$.
By Lemma \ref{lem_d=2_1}, $\bbT \defeq (0,\infty) / \ell_{\rho}(\alpha_j)\bbZ$ is a fundamental domain of $\alpha_j$ in $\bbH_k$. Let $\pi \colon (0,\infty) \to \bbT$ be the projection map. Define $I_\beta \defeq \pi(\hat{I}_{\beta})$.

\begin{lem} \label{lem_d=2_decomp}
We have the following decomposition of $\bbT$: $$\bbT =\pi(\Psi_\rho(\partial_{\infty}V)) \sqcup \left( \bigsqcup_{q=1}^m \bigcup_{\beta \in \mathcal{L}_{j,q}} I_{\beta} \right)$$ where $\pi(\Psi_\rho(\partial_{\infty}V))$ is a finite set and  $\mathcal{L}_{j,q}$ is the subset of the set of double cosets $\langle\alpha_j\rangle \backslash \pi_1\Sigma/\langle \alpha_q \rangle$.
\end{lem}
\begin{proof}
By Lemma \ref{lem_proj_0}, the projection $\Psi_\rho(\partial_\infty V)$ of the limit set onto $[0,\infty]$ is a finite set. Therefore $\pi(\Psi_\rho(\partial_\infty V))$ is a finite set. Since the map $\pi \circ \Psi_\rho$ is continuous, $\bbT$ is equal to the image of the union of non-trivial $\hat{I}_\beta$'s under $\pi \colon (0,\infty) \to \bbT$. Moreover, there is a one-to-one correspondence between the sets $I_\beta$ and the double cosets $\langle \alpha_j \rangle \backslash \pi_1\Sigma/\langle \alpha_q \rangle$. To see this, we observe that
$\hat{I}_{\alpha_j\beta\alpha_j^{-1}} = (\Psi_\rho(\alpha_j\cdot \beta^+), \Psi_\rho(\alpha_j\cdot \beta^-)) = \hat{I}_\beta - \ell_\rho(\alpha_j)$. This completes the proof.
\end{proof}

Let $\nu_{\bbH}$ be the hyperbolic measure on $[0,\infty]$; namely, $\nu_{\bbH}((a,b)) = d_{\bbH}(a,b) = \log_v (b/a)$ for any interval $(a,b)$ with $0<a<b<\infty$. 
\begin{lem} \label{lem_2}
We have $\nu_{\bbH}(\bbT) = \ell_{\rho}(\alpha_j)$ and $$\nu_{\bbH}(I_{\beta}) = \log_v \left| C_{\rho}(\alpha_j^+, \alpha_j^-; w \cdot \alpha_q^+, w \cdot \alpha_q^-) \right|$$ where $w$ and $q$ are as in Lemma \ref{lem_coset_beta}.
\end{lem}
\begin{proof}
By construction of $\bbT$, we have $\nu_{\bbH}(\bbT) = \ell_{\rho}(\alpha_j)$.
We compute
\begin{align*}
	\nu_{\bbH}(I_{\beta}) &= \log_v|\zeta^1(\beta^-)| -  \log_v|\zeta^1(\beta^+)|\\
	& = \log_v |[\infty, 0; \zeta^1(\beta^+), \zeta^1(\beta^-)]|\\
	& =  \log_v \left| C_{\rho}(\alpha_j^+, \alpha_j^-; w \cdot \alpha_q^+, w \cdot \alpha_q^-) \right|
\end{align*}
where $w$ is the unique representative of the double coset corresponding to $I_{\beta}$. Here $w$ is unique if we put an ordering on a symmetric generating set of $\pi_1\Sigma$ and let $w$ be the double coset representative that is the shortest in word length and the largest with respect to the ordering.
\end{proof}

\begin{proof}[Proof of Theorem \ref{thm_realhigherBas} for $d=2$]
Fix a boundary component $\alpha_j$.
By Lemmas \ref{lem_d=2_decomp} and \ref{lem_2}, we have
\begin{align*}
\ell_\rho(\alpha_j) &= \nu_{\bbH}(\bbT) \\
& = \nu_{\bbH}(\pi(\Psi_\rho(\partial_{\infty}V))) +
\sum_{q=1}^m\sum_{w \in \mathcal{L}_{j,q}} \nu_{\bbH}(I_\beta)\\
& = 0+ \sum_{q=1}^m\sum_{w \in \mathcal{L}_{j,q}}  \log_v \left| C_{\rho}(\alpha_j^+, \alpha_j^-; w \cdot \alpha_q^+, w \cdot \alpha_q^-) \right|.
\end{align*}
where the series on the right hand side is a finite sum. The identity is obtained by summing over $\alpha_j$'s.
\end{proof}

\subsection{Proof of Theorem \ref{thm_realhigherBas}} \label{sec_pf_main_thm}
The proof of Theorem \ref{thm_realhigherBas} in general does not use any geometry. Instead, we consider the function $\Phi_\rho$ defined in \eqref{eq_Phi}.

Let $\Sigma$ be as in the previous section. Fix a boundary component $\alpha_j$ of $\Sigma$ and let $\alpha_j^+,\alpha_j^- \in \partial \pi_1\Sigma$ be the attracting and repelling fixed points of $\alpha_j$, respectively.
For $d\ge 2$, let $\rho : \pi_1\Sigma \to \PSL(d,k)$ be projective Anosov. To ease notation, we denote by $\zeta^1 = \zeta^1_{\rho}$. Let $z$ be a reference point in $(\alpha_j^+, \alpha_j^-)$. Define $\Phi_\rho \colon (\alpha_j^+, \alpha_j^-) \to \mathbb{R}$ by
\begin{equation} \label{eq_Phi}
\Phi_\rho(x) \defeq \log_v |C_\rho(\alpha_j^+, \alpha_j^-,x,z)|.
\end{equation}

\begin{lem}
For any $x \in (\alpha_j^+,\alpha_j^-)$, we have
$\Phi_\rho(\alpha_j \cdot x) = \Phi_\rho(x) - \ell_\rho(\alpha_j).$
\end{lem}
\begin{proof}
By definition of $\Phi_\rho$, Lemma \ref{lem_pos} (3) and Lemma \ref{lem_2.8}, we have
\begin{align*}
\Phi_\rho(\alpha_j \cdot x) &= \log_v |C_\rho(\alpha_j^+,\alpha_j^-,\alpha_j \cdot x,z)|\\
&= \log_v \frac{|C_\rho(\alpha_j^+,\alpha_j^-, x,z)|}{|C_\rho(\alpha_j^+,\alpha_j^-, x,\alpha_j \cdot x)|}\\
&=\Phi_\rho(x) - \ell_\rho(\alpha_j).
\end{align*}
The proof is completed.
\end{proof}

\begin{lem} \label{lem_48}
For each interval $\tilde{I}_\beta$, we have
$\Phi_\rho(\beta^+)-\Phi_\rho(\beta^-) = \log_v |C_\rho(\alpha_j^+,\alpha_j^-,g\alpha_q^+,g\alpha_q^-)|$, where $g$ and $q$ are as in Lemma \ref{lem_coset_beta}.
\end{lem}
\begin{proof}
By definition of $\Phi_\rho$ and Lemma \ref{lem_pos} (4) and (3), we have
\begin{align*}
\Phi_\rho(\beta^+)-\Phi_\rho(\beta^-) &= \log_v \frac{|C_\rho(\alpha_j^+,\alpha_j^-,\beta^+,z)|}{|C_\rho(\alpha_j^+,\alpha_j^-,\beta^-,z)|}\\
&= \log_v \left(|C_\rho(\alpha_j^+,\alpha_j^-,\beta^+,z)|\cdot |C_\rho(\alpha_j^+,\alpha_j^-,z,\beta^-)|\right)\\
&= \log_v |C_\rho(\alpha_j^+,\alpha_j^-,\beta^+,\beta^-)|\\
&=\log_v |C_\rho(\alpha_j^+,\alpha_j^-,g\alpha_q^+,g\alpha_q^-)|.
\end{align*}
The proof is completed.
\end{proof}

\begin{lem} \label{lem_finitely_many}
All but finitely many intervals $\widetilde{I}_\beta$ satisfy $|\Phi_\rho(\beta^-)-\Phi_\rho(\beta^+)| = 0$.
\end{lem}
\begin{proof}
As in the proof of Lemma \ref{lem_48}, we have
\begin{align*}
\Phi_\rho(\beta^+)-\Phi_\rho(\beta^-) 
&= \log_v |C_\rho(\alpha_j^+,\alpha_j^-,\beta^+,\beta^-)|.
\end{align*}
Since $\zeta^1 \colon \partial \pi_1\Sigma \to \bbP^1(k)$ is continuous, if $\beta^+$ and $\beta^-$ are close, the cross ratio $C_{\rho}(\alpha_j^+, \alpha_j^-; \beta^+, \beta^-)$ is close to 1. Then the non-Archimedean absolute value gives $|C_{\rho}(\alpha_j^+, \alpha_j^-; \beta^+, \beta^-)| = 1$. Therefore there are only finitely many $\widetilde{I}_\beta$ with $\Psi_\rho(\beta^+) \neq \Psi_\rho(\beta^-)$.
The proof is completed.
\end{proof}

Let $\bbT \defeq \mathbb{R} / \ell_{\rho}(\alpha_j)\bbZ$ be a fundamental domain of $\alpha_j$ in $\bbR$. Let $\pi \colon \bbR \to \bbT$ be the projection map. For each interval $\tilde{I}_\beta$, we denote by $\hat{I}_\beta \subset \mathbb{R}$ the interval $(\Phi_\rho(\beta^-),\Phi_\rho(\beta^+))$ or $(\Phi_\rho(\beta^+),\Phi_\rho(\beta^-))$.

\begin{lem} \label{lem_1}
We have the following decomposition of $\bbT$: $$\bbT =\pi(\Phi_\rho(\partial_{\infty}V)) \sqcup \left( \bigsqcup_{q=1}^m \bigcup_{\beta \in \mathcal{L}_{j,q}} \pi(\hat{I}_{\beta}) \right)$$ where $\pi(\Phi_\rho(\partial_{\infty}V))$ is a finite set and $\mathcal{L}_{j,q}$ is the collection of double cosets $\langle\alpha_j\rangle \backslash \pi_1\Sigma/\langle \alpha_q \rangle$.
\end{lem}
\begin{proof}
The proof follows the same argument as that of Lemma \ref{lem_d=2_decomp}.
\end{proof}

Now we give a proof of Theorem \ref{thm_realhigherBas}.
\begin{proof}[Proof of Theorem \ref{thm_realhigherBas}]
Fix a boundary component $\alpha_j$.
Let $\lambda$ be the Lebesgue measure on $\mathbb{T}$.
We first observe that $\lambda(\bbT) = \ell_\rho(\alpha_j)$ by definition of $\bbT$.
On the other hand, by Lemma \ref{lem_finitely_many}, only finitely many intervals $\pi(\hat{I}_\beta)$ are non-trivial. Combining with Lemma \ref{lem_1}, we have
\begin{align*}
\ell_\rho(\alpha_j) = \lambda(\mathbb{T}) &= \lambda(\pi(\Phi_\rho(\partial_{\infty}V))) + \sum_{q=1}^m \sum_{\beta \in \mathcal{L}_{j,q}} \lambda(\pi(\hat{I}_{\beta}))\\
& = 0+ \sum_{q=1}^m \sum_{\beta \in \mathcal{L}_{j,q}} \log_v |C_\rho(\alpha_j^+,\alpha_j^-,g\alpha_q^+,g\alpha_q^-)|
\end{align*}
where the series on the right hand side is a finite sum. The identity is obtained by summing over $\alpha_j$'s.
\end{proof}

\section{Examples}\label{sec_examples}
In this section, we give examples of projective Anosov representations from the fundamental group of the one-holed torus $\Sigma_{1,1}$ into $\PSL(d,k),d\ge 2$ to demonstrate that the right hand side series of Theorem \ref{thm_realhigherBas} may consists of only positive terms, as well as both positive and negative terms.

\subsection{The case $d=2$}

\begin{example} \label{example_51}
Let $k$ be a non-Archimedean local field whose residue field has cardinality $3$ and $\Sigma_{1,1}$ be the one-holed torus. We give an example of projective Anosov representations from $\pi_1\Sigma_{1,1}$ into $\PSL(2,k)$ where the right hand side series of Theorem \ref{thm_realhigherBas} consists of only positive terms.

Let $\Gamma_1$ be an oriented graph which is a wedge of two oriented circles. Embed $\Gamma_1$ in $\Sigma_{1,1}$ as in Figure \ref{fig_1}. Since $\Sigma_{1,1}$ deformation retracts to the graph $\Gamma_1$, the fundamental group $\pi_1\Sigma_{1,1}$ acts geometrically on the universal cover of $\Gamma_1$, which is the 4-valent tree $\widetilde{\Gamma}_1$.

\begin{figure}[h!] 
\centering
\includegraphics[scale=0.2]{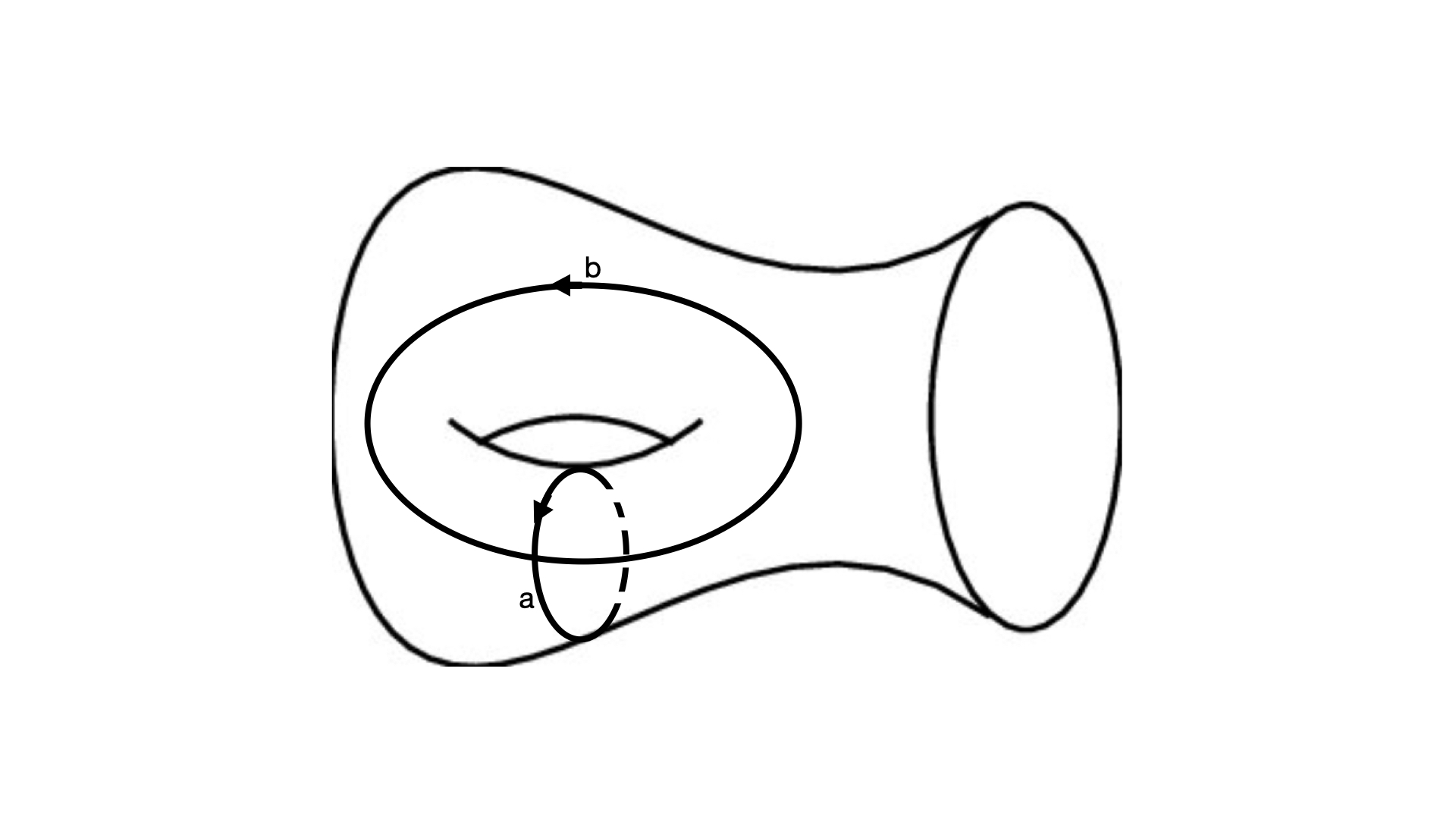}
\caption{$\Gamma_1$ on the one-holed torus $\Sigma_{1,1}$}
\label{fig_1}
\end{figure}

We now describe a projective Anosov representation $\rho_1 \colon \pi_1\Sigma_{1,1} \to \PSL(2,k)$ for which the right hand side series of Theorem \ref{thm_realhigherBas} consists of only positive terms. Let $a,b$ be generators of $\pi_1\Sigma_{1,1}$.
Define
$\rho_1(a)(z) \defeq 3z$ and $\rho_1(b)(z) \defeq \frac{3\frac{z-2}{z-1} -2}{3\frac{z-2}{z-1} -1}$.

The identity in this case is $4 = 1 + 1 + 1 + 1$.
\end{example}

\begin{example} \label{example_52}
Let $k$ be a non-Archimedean local field whose residue field has cardinality $2$.
We now give an example of projective Anosov representations from $\pi_1\Sigma_{1,1}$ into $\PSL(2,k)$ where the right hand side series of Theorem \ref{thm_realhigherBas} consists of both positive and negative terms.

Let $\Gamma_2$ be the oriented dumbbell graph as in Figure \ref{fig_2}. Note that $\Sigma_{1,1}$ deformation retracts to $\Gamma_2$ as there is a homotopy equivalence $\Gamma_2 \to \Sigma_{1,1}$.
Since $\Sigma_{1,1}$ deformation retracts to the graph $\Gamma_2$, the fundamental group $\pi_1\Sigma_{1,1}$ acts geometrically on the universal cover of $\Gamma_2$, which is the trivalent tree $\widetilde{\Gamma}_2$.

\begin{figure}[h!] 
\centering
\includegraphics[scale=0.2]{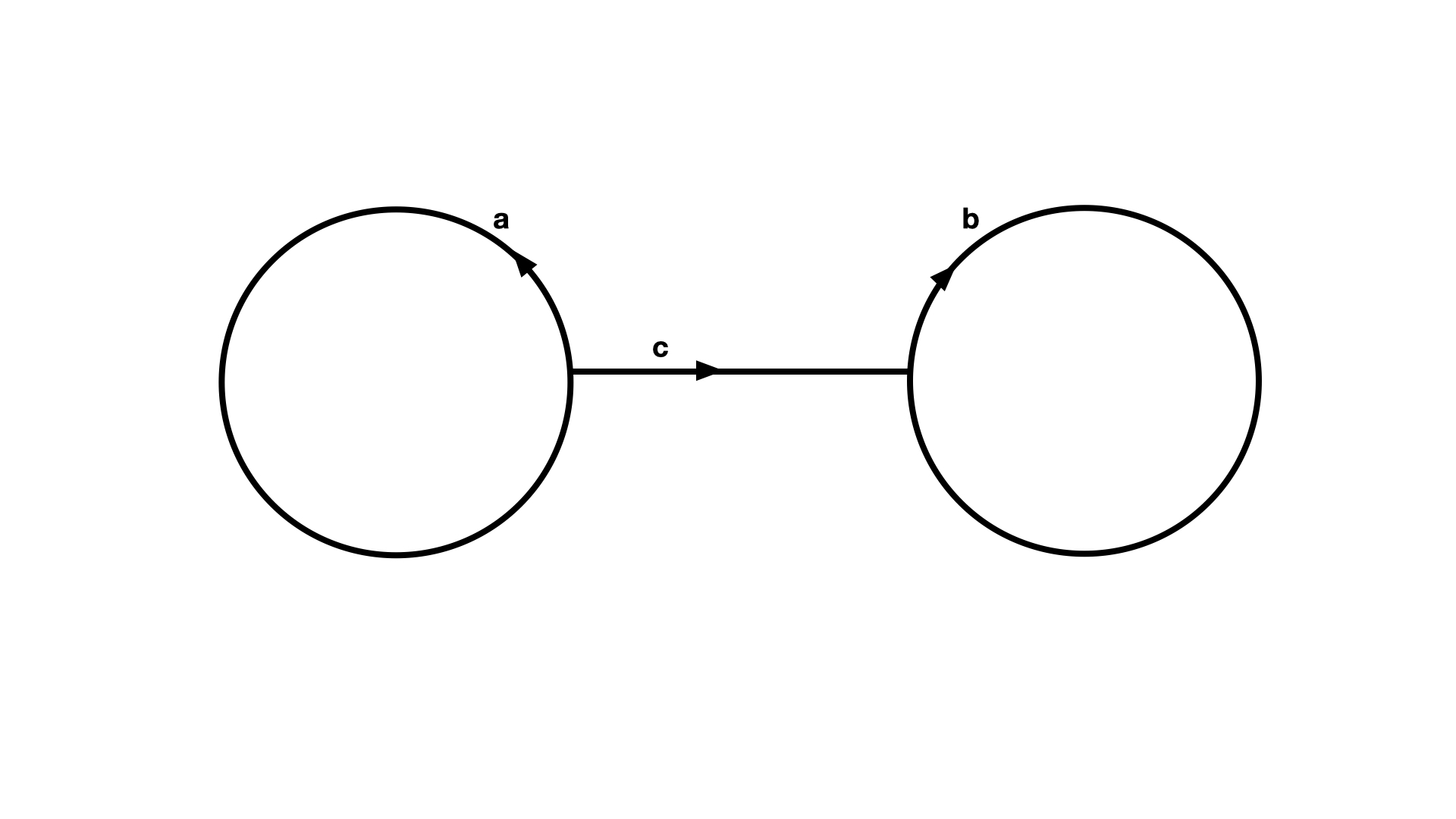}
\caption{The oriented dumbell graph}
\label{fig_2}
\end{figure}

Let $a$ and $cbc^{-1}$ be the generators of the fundamental group of the dumbbell graph. Then the boundary of the one-holed torus is represented by $acbc^{-1}a^{-1}cb^{-1}c^{-1}$. 
Define $\rho_2 \colon \pi_1\Sigma_{1,1} \to \PSL(2,k)$ by
$\rho_2(a)(z) \defeq 2z$ and 
$\rho_2(cbc^{-1})(z) \defeq \frac{5z-3}{z+1}$ (fixing 1 and 3). The quotient of $\widetilde{\Gamma}_2/\langle \rho_2(a), \rho_2(cbc^{-1})\rangle$ is indeed the dumbbell graph.

The identity in this case is
$8 = (3 -1) + (3 - 1) + (3 -1) + (3 - 1)$.
\end{example}

\subsection{The case $d\ge 3$}
Let $d \ge 3$ be an integer.
Recall the Veronese embedding $\iota \colon \PSL(2,k) \to \PSL(d,k)$. Let $x^{d-1},x^{d-1}y,\ldots, y^{d-1}$ be a basis of the real $d$-dimensional vector space of homogeneous polynomials in 2 variables of degree $\le d-1$. Let 
$g = \begin{bmatrix}
    a & b \\
    c& d\\
\end{bmatrix} \in {\rm SL}(2,k)$. Then $\iota(g) \in {\rm SL}(d,\bbR)$ is the matrix whose action on the basis elements is $\iota(g)(x^{d-i-1}y^i) \defeq (ax+by)^{d-i-1}(bx+dy)^i$.

\begin{lem} \label{lem_cr_same}
Let $\rho \colon \pi_1\Sigma \to \PSL(2,k)$ be a projective Anosov representation. Let $x_1,x_2,x_3,x_4$ be 4 distinct points on $\partial \pi_1\Sigma$. Then we have $(d-1)\log_v |C_\rho(x_1,x_2,x_3,x_4)| = \log_v |C_{\iota\circ \rho}(x_1,x_2,x_3,x_4)|$.
\end{lem}
\begin{proof}
We use notations from Lemma \ref{lem_cr_FGW}. Let $\phi \defeq \zeta^1(x_1) \in \bbP(k^d)$ and let $\omega \defeq \zeta^{d-1}(x_3) \in \bbP(k^{d*})$.
Let $\bar{\phi}$ and $\bar{\omega}$ be lifts of $\phi$ and $\omega$. Let $\bar{\phi}_{\iota}$ be the lift of $\zeta^{1}_{\iota \circ \rho}(x_1)$ and $\bar{\omega}_{\iota}$ be the lift of $\zeta^{d-1}_{\iota \circ \rho}(x_3)$.
Then we have
$\bar{\phi}_{\iota}(\bar{\omega}_\iota) = (\bar{\phi}(\bar{\omega}))^{d-1}$. Therefore every term in the formula in Lemma \ref{lem_cr_FGW} is raised to the power $(d-1)$ after composing the representation $\rho$ by $\iota$. Taking log of the cross ratio, the conclusion follows.
\end{proof}

\begin{example}
Let $\rho_1 \colon \pi_1\Sigma_{1,1} \to \PSL(2,k)$ and $\rho_2 \colon \pi_1\Sigma_{1,1} \to \PSL(2,k)$ be the representations as in Examples \ref{example_51} and \ref{example_52} respectively.
By Lemma \ref{lem_cr_same}, we see that $\iota \circ \rho_1 \colon \pi_1\Sigma_{1,1} \to \PSL(d,k)$ is a projective Anosov representation for which the right hand side series of Theorem \ref{thm_realhigherBas} consists of only positive terms.
Similarly, $\iota \circ \rho_2 \colon \pi_1\Sigma_{1,1} \to \PSL(d,k)$ is a projective Anosov representation for which the right hand side series of Theorem \ref{thm_realhigherBas} consists of both positive and negative terms.
\end{example}

\end{document}